\documentclass[twoside]{aiml}

\usepackage{aimlmacro}

\usepackage{graphicx}
\usepackage{amsmath}
\usepackage{amssymb}
\usepackage{amsfonts}
\usepackage{bm}
\usepackage{tikz-cd}
\usepackage{adjustbox}
\usepackage{cleveref}
\usepackage{comment}

\newcommand{\Po}{\mathcal{P}}
\newcommand{\inv}{^{-1}}
\newcommand{\sset}{\subseteq}
\newcommand{\rset}{\supseteq}
\newcommand{\me}{\wedge}
\newcommand{\bigme}{\bigwedge}
\newcommand{\jo}{\vee}
\newcommand{\bigjo}{\bigvee}

\newcommand{\reef}[1]{\reflectbox{$\mathsf{#1}$}}
\newcommand{\ple}[1]{\leq_{\mathsf{#1}}}
\newcommand{\nle}[1]{\leq\reflectbox{$\, _{\mathsf{#1}} \!\!$}}
\newcommand{\upset}{\mathord{\uparrow}}
\newcommand{\dnset}{\mathord{\downarrow}}
\newcommand{\pneg}[1]{\neg_{\mathsf{#1}}}
\newcommand{\nneg}[1]{\neg\reflectbox{$_{\mathsf{#1}}$}}
\newcommand{\pchi}[1]{\chi_{\mathsf{#1}}}
\newcommand{\nchi}[1]{\chi\reflectbox{$_{\mathsf{#1}}$}}
\newcommand{\dle}[1]{\leq\!\!\reflectbox{$_{\mathsf{#1}}$} \! _{\mathsf{#1}}}
\renewcommand{\phi}{\varphi}
\renewcommand{\r}[1]{\mathsf{#1}}
\newcommand{\resim}{\mathord{\sim}}

\DeclareFontFamily{U}{mathx}{}
\DeclareFontShape{U}{mathx}{m}{n}{<-> mathx10}{}
\DeclareSymbolFont{mathx}{U}{mathx}{m}{n}
\DeclareMathAccent{\widehat}{0}{mathx}{"70}
\DeclareMathAccent{\widecheck}{0}{mathx}{"71}

\tikzcdset{every label/.append style = {font = \large}}
\tikzcdset{every cell/.append style = {font = \large}}

\def\lastname{Massas}

\begin{document}

\begin{frontmatter}
  \title{Goldblatt-Thomason Theorems for Fundamental (Modal) Logic}
  \author{Guillaume Massas}\footnote{guillaume.massas@sns.it}
  \address{Scuola Normale Superiore, Pisa}

  \begin{abstract}
  Holliday recently introduced a non-classical logic called Fundamental Logic, which intends to capture exactly those properties of the connectives ``and", ``or" and ``not" that hold in virtue of their introduction and elimination rules in Fitch's natural deduction system for propositional logic. Holliday provides an intuitive relational semantics for fundamental logic which generalizes both Goldblatt's semantics for orthologic and Kripke semantics for intuitionistic logic. In this paper, we further the analysis of this semantics by providing a Goldblatt-Thomason theorem for Fundamental Logic. We identify necessary and sufficient conditions on a class K of fundamental frames for it to be axiomatic, i.e., to be the class of frames satisfying some logic extending Fundamental Logic. As a straightforward application of our main result, we also obtain a Goldblatt-Thomason theorem for Fundamental Modal Logic, which extends Fundamental Logic with standard $\Box$ and $\Diamond$ operators.
  \end{abstract}

 \end{frontmatter}

 \section{Introduction}
The celebrated Goldblatt-Thomason theorem \cite{goldblatt1975axiomatic} for modal logic characterizes when an elementary class $\mathfrak{K}$ of Kripke frames is \textit{axiomatic} (i.e., \textit{modally definable}) in terms of natural closure properties imposed on $\mathfrak{K}$. From a categorical perspective, Goldblatt and Thomason's result relies on transferring results and techniques from universal algebra (particularly Birkhoff's HSP theorem \cite{birkhoff1935structure}) into the setting of relational semantics. This can typically be achieved by \textit{canonical extensions} \cite{Dunn,jonsson1951boolean,jonsson1952boolean} which, in the context of Kripke semantics for modal logic, are closely related to ultrafilter extensions. Because the mathematical core of the Goldblatt-Thomason theorem is algebraic in nature, the result straightforwardly generalizes to many relational semantics for non-classical logics, including Kripke semantics for (modal) intuitionistic logic \cite{degroot2022goldblattthomason,goldblatt2005axiomatic,rodenburg2016intuitionistic}, polarity-based semantics for lattice-expansion logics \cite{conradie2018goldblattthomason}, or the more general setting of coalgebras \cite{kurz2007goldblatt}.\\

Our goal here is to present such a generalization of the Goldblatt-Thomason theorem to Holliday's relational semantics for Fundamental Logic \cite{holliday2022fundamental}. Holliday's semantics generalizes both Goldblatt's semantics for orthologic \cite{goldblatt1974semantic} and Kripke semantics for intuitionistic logic \cite{Kripke}, while arguably preserving the intuitiveness and simplicity of both. Accordingly, this makes it a promising framework for the study of a large class of non-classical logics, towards which our result can be seen as a first step. We proceed as follows. In Section \ref{section2}, we review some background on Fundamental Logic, its algebraic semantics in terms of fundamental lattices and its relational semantics in terms of fundamental frames. We also introduce the relevant notion of morphism between fundamental frames. In Section \ref{section3} we identify the relational duals of subalgebras, homomorphic images and products of fundamental lattices. Section \ref{section4} introduces a particular kind of construction on fundamental frames called the filter extension, which is used to prove our main result, a Goldblatt-Thomason theorem for Fundamental Logic. As a corollary, we also obtain a characterization of those classes of fundamental frames that are axiomatized by a canonical logic extending Fundamental Logic. Finally, Section \ref{section5} generalizes our results to the setting of Fundamental Modal Logic \cite{holliday2024modal}.

 \section{Background} \label{section2}
In this section, we first provide some background on Fundamental Logic and on two semantics for it presented in \cite{holliday2022fundamental}, an algebraic semantics in terms of fundamental lattices and a relational semantics in terms of fundamental frames. We then define a notion of morphism between fundamental frames that is the relational analogue of morphisms between fundamental lattices. We assume some familiarity with basic notions of lattice theory and algebraic logic.

 \subsection{Fundamental Logic}

 Introduced by Holliday in \cite{holliday2022fundamental}, Fundamental Logic is the logic of a propositional language containing connectives $\me$, $\jo$, $\neg$ and constants $\bot$ and $\top$ determined uniquely by the introduction and elimination rules for each connective in Fitch's natural deduction system \cite{fitch1973natural}. Equivalently, letting $\mathcal{L}$ be such a propositional language, $\vdash_{FL}$ is the smallest reflexive and transitive relation on $\mathcal{L} \times \mathcal{L}$ satisfying the following closure conditions:
\vspace{-0.5em}
\begin{align*}
    \bot \vdash \phi &; & \phi \vdash \top &; &\top &\vdash \neg \bot; \\
    \phi \me \psi \vdash \phi&;&\phi \me \psi \vdash \psi &;&\chi \vdash \phi \, \& \, \chi \vdash \psi &\Rightarrow \chi \vdash \phi \me \psi;\\
    \phi \vdash \phi \jo \psi&;&\psi \vdash \phi \jo \psi &;&\phi \vdash \chi \, \& \, \psi \vdash \chi &\Rightarrow \phi \jo \psi \vdash \chi;\\
    \phi \me \neg \phi \vdash \bot &;& \phi \vdash \neg \neg \phi &;& \phi \vdash \psi &\Rightarrow \neg \psi \vdash \neg \phi.
\end{align*}

Fundamental Logic generalizes both the $\to$-free fragment of intuitionistic logic and orthologic. In Fitch's natural deduction system, the former can be recovered from Fundamental Logic by adding the Reiteration rule, and the latter by adding the double negation elimination rule.

Fundamental Logic has a natural algebraic semantics in terms of fundamental lattices. Recall first the following definition.

\begin{definition}
    An antitone map $f: L \to L$ on a lattice $L$ is \textit{dually self-adjoint} if for any $a,b \in L$, $a \leq_L f(b)$ iff $b \leq_L f(a)$.
\end{definition}

\begin{definition}
    A \textit{fundamental lattice} is a pair $(L,\neg)$ such that $L$ is a bounded lattice and $\neg: L \to L$ is a dually self-adjoint map such that $a \me_L \neg a = 0_L$ for any $a \in L$.
\end{definition}

Given a fundamental lattice $(L,\neg)$, a valuation $V$ maps any propositional letter $p \in \mathcal{L}$ to some $a \in L$, and is then recursively extended to any formula in $\mathcal{L}$ by using the operations on $(L, \neg)$ in the obvious way. For any two formulas $\phi,\psi \in \mathcal{L}$, $\psi$ is an algebraic consequence of $\phi$ iff $V(\phi) \leq_L V(\psi)$ for any fundamental lattice $(L, \neg)$ and any valuation $V$ on $L$. As shown in Holliday \cite{holliday2022fundamental}, Fundamental Logic is sound and complete with respect to the algebraic semantics thus obtained. 

 \subsection{Fundamental Frames}

     By a relational frame, we simply mean a pair $(X,\r R)$ such that $\r R$ is a co-serial relation on $X$, meaning that for any $x \in X$ there is $x' \in X$ such that $x' \r R x$. Points in a relational frame can be viewed as partial states of information, situations, or positions in discourse, at which propositions may be either \textit{accepted} or \textit{rejected}. The relation $\r R$ can be interpreted as a relation of \textit{openness} between such states, where $x \r R y$ is interpreted as ``$x$ is open to $y$''. In what follows, we write $\reef {\r R}$ for the converse of the relation $\r R$.

     Given a relation frame $(X,\r R)$, the relation $\r R$ induces two antitone operations $\pneg R, \nneg R : \Po(X) \to \Po(X)$ given by $\pneg R A = \{x \in X \mid \forall x' \r R x: x' \notin A\}$ and $\nneg R A = \{x \in X \mid \forall x' \reef  R x: x' \notin A\}$ for any $A \sset X$. Clearly, these two maps form a contravariant adjunction, so they induce two anti-isomorphic complete lattices $\pchi R(X)$ and $\nchi R(X)$ with domains $\{\pneg R A \mid A \sset X\}$ and $\{ \nneg R A \mid A \sset X\}$ respectively, which are the fixpoints of the operations $\pneg R \nneg R$ and $\nneg R \pneg R$ respectively. The composition of the two maps $\pneg R$ and $\nneg R$ yields a closure operator $C_{\r R}$, which can be explicitly described as \[C_{\r R}(A) = \{x \in X \mid \forall y \r R x \exists z \reef{R} y: z \in A\}\] for any $A \sset X$. Finally, we call $\pchi R(X)$ the \textit{positive algebra} of the frame $(X,\r R)$, and $\nchi R$ its \textit{negative algebra}. 

The following relations can always be defined from an openness relation $\r R$, and will be repeatedly used throughout.

\begin{definition}
        Let $(X,\r R)$ be a relational frame. For any two elements $x,x' \in X$, $x$ \textit{positively refines} $x'$ (noted $x \ple R x'$) if for any $z \in X$: $z\r Rx \Rightarrow z\r Rx'$, and $x$ \textit{negatively refines} $x'$ (noted $x \nle R x'$) if for any $z \in X$: $x \r Rz \Rightarrow x' \r Rz$.
\end{definition}

We will also often appeal to some straightforward facts about relational frames whose proofs we omit.

\begin{fact} \label{facts}
    The following hold for any relational frame $(X,R)$:
    \begin{enumerate}
        \item For any $A \sset X$, $A \in \pchi R(X)$ iff $\forall x \in X:$ $x \in A \Leftrightarrow \forall x' \r R x \exists y \reef R x': y \in A$;
        \item For any $A \sset X$, $A \in \nchi R(X)$ iff $\forall x \in X$: $x \in A \Leftrightarrow \forall x' \reef R x \exists y \r R x': y \in A$;
        \item For any $x \in X$, $\dnset_{\r R}(x) = \{y \in X \mid y \ple R x\}$ and $\overline{\r R}(x) = \{y \in X \mid \neg x \r R y\}$ are elements in $\pchi R(X)$;
        \item For any $x \in X$, $\dnset\reef {_R}(x) = \{y \in X \mid y \nle R x\}$ and $\overline{\reef R}(x) = \{y \in X \mid \neg x \reef R y\}$ are elements in $\nchi R(X)$;
        \item $\ple R$ and $\nle R$ are preorders on X;
        \item For any $x, x' \in X$, $x \ple R x'$ iff for any $A \in \pchi R(X)$, $x' \in A$ implies $x \in A$;
        \item For any $x, x' \in X$, $x \nle R x'$ iff for any $B \in \nchi R(X)$, $x' \in B$ implies $x \in B$.

    \end{enumerate}
\end{fact}

For any relational frame $(X,\r R)$, its positive algebra $\pchi R (X)$ can be equipped with the unary antitone map $\pneg R$. In order to ensure that the resulting pair $(\pchi R(X), \pneg R)$ is a fundamental lattice, we need to impose two conditions on $\r R$:

\begin{definition}
Let $(X,\r R)$ be a relational frame. Let $\r R(x)$ and $\reef R (x)$ be the sets $\{y \in X \mid x \r R y\}$ and $\{y \in X \mid  y \r R x\}$ respectively.
\begin{itemize} 
    \item $\r R$ is \textit{pseudo-reflexive} if $\reef R(x) \cap \dnset_{\r R}(x) \neq \emptyset$ for any $x \in X$;
    \item $R$ is \textit{pseudo-symmetric} if for any $x, x' \in X$, $x' \in \r R(x)$ implies $\reef R (x) \cap \dnset_R(x') \neq \emptyset$.
\end{itemize}

A \textit{fundamental frame} is a relational frame $(X, \r R)$ such that $\r R$ is pseudo-reflexive and pseudo-symmetric.
\end{definition}

Fundamental frames provide a relational semantics for fundamental logic in a straightforward way. Given a fundamental frame $(X, \r R)$, a valuation $V$ maps any propositional letter $p \in \mathcal{L}$ to some $A \in \pchi R(X)$, and is recursively extended to all formulas $\phi \in \mathcal{L}$ as follows:
\begin{itemize}
    \item $V(\neg \phi) = \pneg R V(\phi)$;
    \item $V(\phi \me \psi) = V(\phi) \cap V(\psi)$;
    \item $V(\phi \jo \chi) =C_{\r R} (V(\phi) \cup V(\psi))$.
\end{itemize}

This ensures that formulas are always evaluated as elements in the positive algebra of a fundamental frame. As usual, given a fundamental frame $(X, \r R)$, we write $\phi \models_{(X,\r R)} \psi$ if $V(\phi) \sset V(\psi)$ for any valuation $V$ on $(X,\r R)$. The following establishes the soundness of this semantics for Fundamental Logic.

\begin{theorem}[\cite{holliday2022fundamental}, Prop.~4.14] \label{sndthm}
    For any relational frame $(X, \r R)$, $\pchi R(X)$ is a fundamental lattice iff $(X, \r R)$ is a fundamental frame.
\end{theorem}

As is standard in relational semantics, completeness is established via a canonical frame construction. Since this construction plays a central role in this paper, we briefly review it now.

\begin{definition} \label{candef}
    Let $(L, \neg)$ be a fundamental lattice. The \textit{canonical frame} of $(L, \neg)$ is the relational frame $\digamma(L) = (X, \r R)$ given by the following data:
    \begin{itemize}
        \item $X$ is the set of all pairs $(F,I)$ such that $F$ and $I$ are a proper filter and a proper ideal on $L$ respectively, and $F \sset \neg\inv[I]$.
        \item For any $(F,I),(G,J) \in X$, $(F,I)\r R(G,J)$ iff $G \cap I = \emptyset$.
    \end{itemize}    
\end{definition}

It is straightforward to verify that, if $(X, \r R) = \digamma(L)$ for some fundamental lattice $(L, \neg)$, then the positive refinement relation $\ple R$ is given by converse inclusion on filters and the negative refinement relation $\nle R$ is given by converse inclusion on ideals. In other words, for any two points $(F,I),(G,J) \in X$, $(F,I) \ple R (G,J)$ iff $F \rset G$, and $(F,I) \nle R (G,J)$ iff $I \rset J$. Moreover, any fundamental lattice embeds into the positive algebra of its canonical frame via a standard Stone-like map.

\begin{theorem}[\cite{holliday2022fundamental}, Thm.~B.7 ]
    For any fundamental lattice $(L, \neg)$, $\digamma(L) = (X,R)$ is a fundamental frame, and the map $\widehat{\cdot} : (L, \neg) \to (\chi_{\r R}(X), \pneg R)$ given by $a \mapsto \{(F,I) \in X \mid a \in F\}$ is a lattice embedding such that $\widehat{\neg a} = \pneg R \widehat{a}$ for all $a \in L$.
\end{theorem}

Given a class $\mathfrak{K}$ of fundamental frames, we let $Log(\mathfrak{K})$ be the set $\{ \phi \in \mathcal{L} \mid \forall (X,\r R) \in \mathfrak{K}: \top \models_{(X,\r R)} \phi\}$. Similarly, given some $\Gamma \sset \mathcal{L}$, we let $Mod(\Gamma)$ be the class of fundamental frames $\{(X,\r R) \mid \forall \phi \in \Gamma: \top \models_{(X,\r R)} \phi\}$. As usual, $Log$ and $Mod$ form a Galois connection between the set of subsets of $\mathcal{L}$ and the class of all classes of fundamental frames, both ordered by inclusion.
    
 \subsection{$F$-Morphisms}

 In this section, we introduce the notion of a fundamental morphism between fundamental frames. We show first that such morphisms induce fundamental lattice homomorphisms between positive algebras in a natural way, before establishing that fundamental lattice homomorphisms induce fundamental morphisms between canonical frames.

\begin{definition}
    Let $(X,\r R)$ and $(Y,\r S)$ be two fundamental frames. A \textit{fundamental morphism} ($f$-morphism for short) is a map $h: X \to Y$ satisfying the following four properties for any $x,x' \in X$ and $y,y' \in Y$:

    \begin{enumerate}
        \item $x \r R x'$ implies $h(x) \r S h(x')$;
        \item $h(x) \r S y$ implies $\exists x' \in X: x \r R x'$ and $h(x') \ple S y$;
        \item $y \r S h(x)$ implies $\exists x' \in X: x' \r R x$ and $h(x') \nle S y$;
        \item $y \r S h(x)$ implies $\exists x'' \in X: x'' \r R x$ and $h(x'') \ple S y$.
    \end{enumerate}
\end{definition}

The following is a diagrammatic representation of the conditions required on $f$-morphisms. Single arrows are labelled according to which relation they represent, and double arrows are implications. From left to right, each diagram corresponds to conditions $1$, $2$ and $3-4$ respectively.

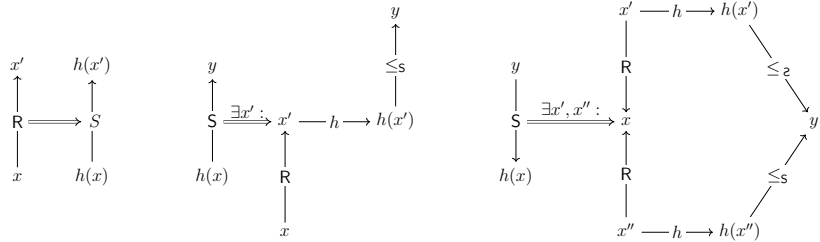
\adjustbox{width=0.9\textwidth,center}{%
\begin{tikzcd}
	&&&&&& { y} &&&& { x'} && { h(x')} \\
	{ x'} & { h(x')} && { y} &&&&& { y} && {\phantom{0}} \\
	{\phantom{0}} & {\phantom{0}} && {\phantom{0}} & { x'} && { h(x')} && {\phantom{0}} && { x} &&& {y} \\
	{x} & {h(x)} && {h(x)} &&&&& {h(x)} \\
	&&&& {x} &&&&&& {x''} && {h(x'')}
	\arrow["{\r S}"{description}, from=4-4, to=2-4]
	\arrow["{\leq_{\r S}}"{description}, from=3-7, to=1-7]
	\arrow["{\r R}"{description}, from=5-5, to=3-5]
	\arrow["{h}"{description}, from=3-5, to=3-7]
	\arrow["{\r R}"{description}, from=4-1, to=2-1]
	\arrow["{\ S}"{description}, from=4-2, to=2-2]
	\arrow["{\r R}"{description}, from=1-11, to=3-11]
	\arrow["{\r S}"{description}, from=2-9, to=4-9]
	\arrow["{h}"{description}, from=1-11, to=1-13]
	\arrow["{ \reef{_S \!\geq}}"{description}, from=1-13, to=3-14]
	\arrow[Rightarrow, from=3-1, to=3-2]
	\arrow["{\exists x':}", Rightarrow, from=3-4, to=3-5]
	\arrow["{\r R}"{description}, from=5-11, to=3-11]
	\arrow["{h}"{description}, from=5-11, to=5-13]
	\arrow["{\leq_{\mathsf{S}}}"{description}, from=5-13, to=3-14]
	\arrow["{\exists x',x'':}", Rightarrow, from=3-9, to=3-11]
\end{tikzcd}}

\begin{definition}
    Let $(L, \neg), (M, \resim)$ be two fundamental lattices. A \textit{fundamental lattice homomorphism} is a lattice homomorphism $f: (L, \neg) \to (M, \resim)$ such that $f(\neg a) = \resim f(a)$ for any $a \in L$.
\end{definition}

\begin{lemma} \label{fmorlma}
    For any $f$-morphism $h: (X,\r R) \to (Y,\r S)$, the map $\chi(h) : \pchi S(Y) \to \pchi R(X)$ given by $A \mapsto h\inv[A]$ is a fundamental lattice homomorphism.
\end{lemma}

\begin{proof}
We claim that for any $A \sset Y$, $C_{\r R}(h\inv[A]) = h\inv[C_{\r S}(A)]$. It is routine to check that this implies that $\chi(h): \pchi S(Y) \to \pchi R(X)$ is well-defined and a lattice homomorphism. Fix some $A \sset Y$. For the right-to-left inclusion, assume that $h(x) \in C_{\r S}(A)$, and let $x' \r R x$. Then by property $1$ of $f$-morphisms, $h(x') \r S h(x)$. Since $h(x) \in C_{\r S}(A)$, this means that there is $y \in A$ such that $h(x') \r S y$. By property $2$ of $f$-morphisms, there is $x'' \in X$ such that $x' \r R x''$ and $h(x'') \ple S y$. But the latter implies that $h(x'') \in A$ and thus that $x'' \in h\inv[A]$. This shows that $x \in C_{\r R}(h\inv[A])$.

For the converse inclusion, assume that $x \in C_{\r R}(h\inv[A])$, and let $y \r S h(x)$. By property $3$ of $f$-morphisms, there is $x' \in X$ such that $x' \r R x$ and $h(x') \nle S y$. Since $x \in C_{\r R}(h\inv[A])$, there is $x'' \in h\inv[A]$ such that $x' \r R x''$. By property $1$ of $f$-morphisms, we have that $h(x')\r S h(x'')$. As $h(x') \nle S y$, we also have $y \r S h(x'')$. But this shows that $x \in C_{\r S}(h\inv[A])$, as desired.

Finally, let us show that $\chi(h)$ preserves fundamental complements. Fix some $A \in \pchi S(Y)$. We need to show that for any $x \in X$, \[x \in h\inv[\neg_{\r S}A]\Leftrightarrow x \in \neg_{\r R} h\inv[A].\] Suppose first that $h(x) \in \neg_{\r S}A$, and let $x' \r R x$. Then $h(x') \r S h(x)$, hence $h(x') \notin A$. This shows the left-to-right direction of the biconditional. For the converse, assume that $x \in \neg_{\r R} h\inv[A]$, and let $y \r S h(x)$. By property $4$ of $f$-morphisms, there is $x'' \in X$ such that $x'' \r R x$ and $h(x'') \ple S y$. Since $x \in \neg_{\r R} h \inv[A]$, we have that $h(x'') \notin A$, from which it also follows that $y \notin A$. This completes the proof.
\end{proof}

One can verify that the definition of an $f$-morphism is slightly stronger than what would need to be required of a map $h$ to ensure that its inverse image is a fundamental lattice homomorphism.\footnote{See Holliday \cite{holliday2022fundamental}, footnote 15 for more on this.} Nonetheless, our definition is easy to state and general enough, as established by the following lemma.

\begin{lemma} \label{canmorlma}
    Let $f: (L, \neg) \to (M, \resim)$ be a fundamental lattice homomorphism. Then the map $(F,I) \mapsto (f\inv[F],f\inv[I])$ yields a $f$-morphism $\digamma(f): \digamma(M) \to \digamma(L)$ such that $\chi(\digamma(f))(\widehat{a}) = \widehat{f(a)}$ for any $a \in L$.
\end{lemma}

\begin{proof}
    Given a pair $(F,I)$ in $\digamma(M)$, let $\digamma(f)(F,I) = (f\inv[F], f\inv[I])$. To see that $\digamma(f)(F,I) \in \digamma(L)$, it is enough to notice that \[a \in f\inv[F] \Rightarrow f(a) \in F \Rightarrow \mathord{\resim f(a)} \in I \Rightarrow f(\neg a) \in I \Rightarrow \neg a \in f\inv[I].\] Let us now check that $\digamma(f)$ is a $f$-morphism. In what follows, we will write points in $\digamma(M)$ and $\digamma(L)$ as pairs of the form $x : (x_F, x_I)$. Let $\r R$ be the canonical relation on $\digamma(M)$ given by $x\r R x'$ iff $x_I \cap x'_F = \emptyset$, and $\r S$ the canonical relation on $\digamma(L)$ defined similarly.

    \begin{enumerate}
        \item For condition $1$, we claim that for any $x, x' \in \digamma(M)$, $x_I \cap x'_F = \emptyset$ implies that $\digamma(f)(x)_I \cap \digamma(f)(x')_F = \emptyset$. This clearly follows from the fact that $\digamma(f)(x)_I = f\inv[x_I]$ and that $\digamma(f)(x')_F = f\inv[x'_F]$. As a consequence, we have that $x\r Rx'$ implies $\digamma(f)(x)\r S\digamma(f)(x')$.
        \item Now suppose that $\digamma(f)(x)\r Sy$, i.e., $f\inv[x_I] \cap y_F = \emptyset$, and let $x'=(\upset f[y_F],I')$, where $I' = \{\neg c \mid c \in \upset f[y_F]\}$. To verify that $x' \in \digamma(M)$, it is enough to show that $0 \notin \upset f[y_F]$. But if there is $c \in y_F$ such that $f(c) \leq 0$, then $c \in f\inv[x_I]$, contradicting our assumption. 
        Note that $y_F \sset f\inv[\upset f [y_F]]$, from which it follows that $\digamma(f)(x') \ple S y$. Hence we only need to verify that $x \r R x'$, i.e., that $x_I \cap \upset f[y_F] = \emptyset$. But if there are $a \in x_I, b \in y_F$ with $f(b) \leq a$, it follows that $b \in f\inv[x_I] \cap y_F$, contradicting our assumption. Hence $x'$ is the required point in $\digamma(M)$.
        \item For condition $3$, suppose that $y \r S \digamma(f)(x)$, i.e., $y_I \cap f\inv[x_F] = \emptyset$. Let $x' = (\{1\}, \dnset f[y_I])$. Let us first verify that $x' \in \digamma(M)$. Clearly, $x'_F$ is a filter and $x'_I$ is an ideal, and if there is $b \in y_I$ such that $1 \leq f(b)$, then $b \in f\inv[x_F]$, contradicting our assumption. Moreover, $y_I \sset f\inv[\dnset f[y_I]]$, which means that $\digamma(f)(x') \nle S y$. Hence it only remains to check that $x' \r R x$, i.e., that $\dnset f[y_I]] \cap x_F = \emptyset$. To see this, suppose that there is $a \in x_F$ and $b \in y_I$ such that $a \leq f(b)$. Then $b \in f\inv[x_F]$, contradicting our assumption.
    \item Finally, we check condition $4$. Once again, suppose that $y \r S \digamma(f)(x)$, i.e., $y_I \cap f\inv[x_F] = \emptyset$. Let $x'' = (\upset f[y_F],I')$, where $I' = \{\neg c \mid c \in \upset f[y_F]\}$. To see that $x'' \in \digamma(M)$, it is enough to verify that $ 0 \notin \upset f[y_F]$. But if there is $c \in y_F$ such that $0 = f(c)$, then $f(\neg c) = \neg f(c) = 1$, so $\neg c \in y_I \cap f\inv[x_F]$, contradicting our assumption. 
    Once again, it is easy to verify that $\digamma(x'') \ple S y$, so we only check that $x'' \r Rx$. This amounts to verifying that $\{\neg c \mid c \in \upset f[y_F]\} \cap x_F = \emptyset$. Suppose, towards a contradiction, that there is $a \in x_F$ and $c \in \upset f[y_F]$ such that $a \leq \neg c$. Then there is $b \in y_F$ such that $ f(b) \leq c$, which implies that $a \leq \neg c \leq \neg f(b) \leq f(\neg b)$, and thus $\neg b \in f\inv[x_F]$. At the same time, $b \in y_F$ implies $\neg b \in y_I$, which means that $y_I \cap f\inv[x_F] \neq \emptyset$, contradicting our assumption.
    \end{enumerate}

    We conclude by showing that $\chi(\digamma(f))(\widehat{a}) = \widehat{f(a)}$ for any $a \in L$. It is enough to show that for any $x \in \digamma(M)$, $f(a) \in x_F$ iff $x \in \digamma(f)\inv[\widehat{a}]$. But the latter is equivalent to $a \in f\inv[x_F]$, which clearly holds iff $f(a) \in x_F$. This concludes the proof.
\end{proof}

The results gathered so far can be conveniently summed up with the following categorical perspective. Let $\mathbf{FL}$ be the category of fundamental lattices and fundamental lattice homomorphisms between them, and let $\mathbf{FFrm}$ be the category of fundamental frames and $f$-morphisms between them. Then we have two contravariant functors $\digamma: \mathbf{FL} \to \mathbf{FFrm}$ and $\chi: \mathbf{FFrm} \to \mathbf{FL}$ defined as follows:
\begin{itemize}
    \item For any $L \in \mathbf{FL}$, $\digamma(L)$ is the canonical frame of $L$ as defined in Definition \ref{candef};
    \item For any fundamental lattice homomorphism $f: L \to M$, $\digamma(f) : \digamma(M) \to \digamma(L)$ is given by $\digamma(h)(x_F,x_I) = (f\inv[x_F],f\inv[x_I])$ for any $(x_F,x_I) \in \digamma(M)$, and is a $f$-morphism by Lemma \ref{canmorlma};
    \item For any $(X,\r R) \in \mathbf{FFrm}$, $\chi(X,\r R) = (\pchi R(X),\pneg R)$, which is a fundamental lattice by Theorem \ref{sndthm};
    \item For any $f$-morphism $h: (X,\r R) \to (Y,\r S)$, $\chi(h): \chi(Y,\r S) \to \chi(X,\r R)$ is given by $\chi(h)(A) = h\inv[A]$ for any $A \in \pchi S(Y)$ and is a fundamental lattice homomorphism by Lemma \ref{fmorlma}.
\end{itemize}

Finally, for any fundamental lattice $(L,\neg)$, the map $\widehat{\cdot} : L \to \chi\digamma(L)$ is an embedding, a fact that we will use several times below.

 \section{Subframes, Dense Images and Coproducts} \label{section3}

 In this section, we identify the relational duals of subalgebras, homomorphic images and products of fundamental lattices. For the first two, we will give necessary and sufficient conditions on a $f$-morphism $h$ for its dual $\chi(h)$ to be injective (resp. surjective), as well as necessary and sufficient conditions on $\digamma(f)$ when a fundamental homomorphism $f$ is injective (resp. surjective). Because the class of relational frames we consider is larger than the class of dual frames of fundamental lattices, the two conditions do not coincide. However, as we shall see in the next section, the characterization given here will be enough to yield a version of the Goldblatt-Thomason theorem.

 We start by identifying when the dual of a $f$-morphism is injective or surjective.

 \begin{definition} \label{defdenemb}
Let $h: (X, \r R) \to (Y,\r S)$ be a $f$-morphism.

\begin{itemize}
    \item $h$ is \textit{dense} if for any $y, y' \in Y$:

     \[y' \r S y \Rightarrow \exists x \in X: h(x) \ple S y \text{ and } y' \r S h(x).\]

     \item $h$ is an \textit{embedding} if for any $x, x' \in X$:

     \[h(x) \r S h(x') \Rightarrow \exists z \in X: x\r R z \text{ and } z \ple R x'.\]
     \end{itemize}
 \end{definition}

 The definitions of dense $f$-morphisms and embeddings can be given mirroring diagrammatic representations, with dense $f$-morphisms represented on the left and embeddings on the right:
 
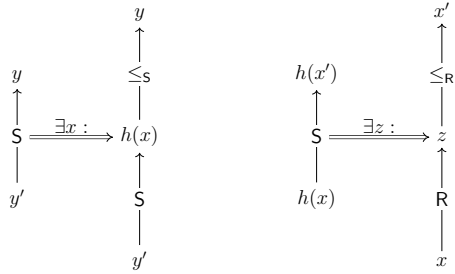
\adjustbox{width=0.5\textwidth,center}{%
\begin{tikzcd}
	&& {y} &&&&& {x'} \\
	{y} &&&&& {h(x')} \\
	{\phantom{0}} && {h(x)} &&& {\phantom{0}} && z \\
	{y'} &&&&& {h(x)} \\
	&& {y'} &&&&& {x}
	\arrow["{\r S}"{description}, from=4-1, to=2-1]
	\arrow["{\exists x:}", Rightarrow, from=3-1, to=3-3]
	\arrow["{\ple S}"{description}, from=3-3, to=1-3]
	\arrow["{\r S}"{description}, from=5-3, to=3-3]
	\arrow["\r S"{description}, from=4-6, to=2-6]
	\arrow["{\exists z:}", Rightarrow, from=3-6, to=3-8]
	\arrow["\r R"{description}, from=5-8, to=3-8]
	\arrow["{\ple R}"{description}, from=3-8, to=1-8]
\end{tikzcd}}
\vspace{0.2em}

Let us show first that dense $f$-morphisms induce injective homomorphisms.

\begin{lemma} \label{injlma}
    For any $f$-morphism $h: (X, \r R) \to (Y,\r S)$, $\chi(h)$ is injective if and only if $h$ is dense.
\end{lemma}
\begin{proof} Assume first that $h$ is dense, and fix $A, B \in \pchi S(Y)$. We claim that $h\inv[A] \sset h\inv[B]$ implies $A \sset B$. To see this, fix $y \in A$ and $y' \in Y$ such that $y'\r Sy$. Since $h$ is dense, there is $x \in X$ such that $y' \r S h(x) \ple S y$. Since $y \in A$, we also have that $h(x) \in A$, hence $x \in h\inv[A] \sset h\inv[B]$, from which it follows that $h(x) \in B$. But this implies that $y \in  C_{\r R}(B) = B$.

Conversely, let us now assume that $h$ is not dense. This means that we have $y, y' \in Y$ such that $y'\r Sy$ and $h(x) \ple S y$ implies $\neg y' \r S h(x)$ for any $x \in X$. Now consider $A = \dnset_S(y)$ and $B = \overline{\r S}(y')$. By choice of $y$ and $y'$, we have that $h\inv[A] \sset h\inv[B]$, but also note that $y \in A \setminus B$. Hence $h$ is not injective. \end{proof}

Let us now show that embeddings induce surjective homomorphisms.

\begin{lemma} \label{surjlma}
    For any $f$-morphism $h: (X, \r R) \to (Y, \r S)$, $\chi(h)$ is surjective if and only if $h$ is an embedding.
\end{lemma}

\begin{proof}
    Suppose first that $h$ is an embedding, and fix $A \in \pchi R(X)$. We claim that $A = h\inv h[A]$. Note that this implies that \[\chi(h)(C_S(h[A])) = h\inv[C_S(h[A])] = C_R(h\inv h[A]) = C_R(A) = A,\] and thus that $\chi(h)$ is surjective. For the proof of the claim, note first that the inclusion $A \sset h\inv h[A]$ is clear. For the converse, assume $h(x) \in h[A]$ for some $x \in X$. This means that $h(x) = h(x')$ for some $x' \in A$. Now let $z \r Rx$. This implies that $h(z)\r Sh(x')$, so, since $h$ is an embedding, there is $w \in X$ with $z\r Rw \ple R x'$. But this means that $w \in A$, and therefore that $x \in C_R(A) = A$.

    Conversely, assume that $\chi(h)$ is surjective and fix $x, x' \in X$ such that $h(x) \r S h(x')$. Since $h$ is an $f$-morphism, there is $z \in X$ such that $x\r Rz$ and $h(z) \ple S h(x')$. Now since $\chi(h)$ is surjective, $\dnset_{\r R}(x') = h\inv[A]$ for some $A \in \pchi S(Y)$. Since $h(z) \ple S h(x')$ and $h(x') \in A$, it follows from Fact \ref{facts}$(vi)$ that $h(z) \in A$, and therefore $z \in \dnset_{\r R}(x')$. But this means that $z \ple R x'$, and therefore that $h$ is an embedding.
\end{proof}

Dense $f$-morphisms therefore induce injective fundamental homomorphisms, and embeddings induce surjective fundamental homomorphisms. Interestingly, this means that the dual of an $f$-morphism may be an isomorphism without $f$ itself being an isomorphism of fundamental frames. This situation is not uncommon. A similar phenomenon occurs for dense embeddings in the so-called ``forcing duality'' between complete Boolean algebras and separative posets \cite{kunen2014set}, and in its generalization to complete lattices given by the b-frame duality presented in \cite{massas2023b}. 

Starting from fundamental homomorphisms instead of $f$-morphisms, let us now identify conditions on $\digamma(h)$ that are equivalent to $h$ being injective or surjective for a fundamental homomorphism $h$.

\begin{definition}
    Let $h: (X,\r R) \to (Y,\r S)$ be a $f$-morphism. Then $h$ is \textit{strongly dense} if for any $y \in Y$, there is $x \in X$ such that $h(x) \ple S y$ and $y \nle S h(x)$, and it is a \textit{strong embedding} if for any $x,x' \in X$, $h(x) \r S h(x')$ implies $x \r R x'$.
\end{definition}

Clearly, these two conditions are strengthenings of those in Definition \ref{defdenemb}. As the result below establishes, they are the frame correspondents of injectivity and surjectivity in the case of canonical frames.

\begin{lemma} \label{replma2}
    Let $f: (L, \neg) \to (M, \resim)$ be a fundamental lattice homomorphism. Then:
    \begin{enumerate}
        \item $f$ is injective iff $\digamma(f)$ is strongly dense.
        \item $f$ is surjective iff $\digamma(f)$ is a strong embedding.
    \end{enumerate}
\end{lemma}

\begin{proof}
 We prove both items in turn. For convenience, we write $\r R$ for the relation on $\digamma(M)$ and $\r S$ for the relation on $\digamma(L)$. 

 \begin{enumerate}
     \item Assume first that $f$ is injective, and let $x \in \digamma(L)$. Let $y = (\upset f[x_F], \dnset f[x_I])$. Note that $f(a) \leq b$ for some $a \in F$ implies that $\resim b \leq \resim f(a) = f(\neg a)$, which shows that $b \in \upset f[x_F]$ implies $\resim b \in \dnset f[x_I]$. Hence $y \in \digamma(M)$. Moreover, since $f$ is injective, we have that $x_F = f\inv[\upset f[x_F]]$ and $x_I = f\inv[\upset f[x_I]]$, which shows that $y \ple S x$ and $x \nle S y$. Conversely, suppose that $\digamma(f)$ is strongly dense, and let $a,b \in L$ such that $a \nleq b$. Since strong density implies density, by Lemma \ref{injlma}, we have that $\chi\digamma(f)$ is injective. Moreover, $a \nleq b$ implies that $\widehat{a} \nsubseteq \widehat{b}$. Hence $\chi(\digamma(f))(\widehat{a}) \nsubseteq \chi(\digamma(f))(\widehat{b})$. But by Lemma \ref{canmorlma}, we have that $\chi(\digamma(f))(\widehat{a}) = \widehat{f(a)}$ and $\chi(\digamma(f))(\widehat{b}) = \widehat{f(b)}$, which means that $\widehat{f(a)} \nsubseteq \widehat{f(b)}$. But this implies that $f(a) \nleq f(b)$, establishing that $f$ is injective.

     \item Assume first that $f$ is surjective, and let $x, x' \in \digamma(M)$ be such that $\digamma(f)(x)\r S\digamma(f)(x')$, i.e., $f\inv[x_I] \cap f\inv[x'_F] = \emptyset$. Suppose towards a contradiction that there is $a \in x_I \cap x'_F$. Then $a = f(b)$ for some $b \in L$, so $b \in f\inv[x_I] \cap f\inv[x'_F]$, contradicting our assumption. Hence $x\r Rx'$. Conversely, suppose now that $f$ is not surjective, and let $a \in M$ be such that $f(b) \neq a$ for all $b \in L$. Note in particular that $a \neq 1$, so $x := ({1}, \dnset a)$ and $x':=(\upset a, \dnset \neg a)$ are points in $\digamma(M)$. Clearly we have $\neg x\r Rx'$. However, by choice of $a$, $f\inv[\dnset a] \cap f\inv[\upset a] = \emptyset$, from which it follows that $\digamma(f)(x)\r S \digamma(f)(x')$. Hence $\digamma(f)$ is not strongly dense.
 \end{enumerate}
\end{proof}

The previous results motivate the following definitions.

\begin{definition}
    Let $(X,\r R)$ and $(Y,\r S)$ be fundamental frames. Then $(X,\r R)$ is a \textit{subframe} of $(Y,\r S)$ if there is an embedding $h: (X,\r R) \to (Y,\r S)$, and it is a \textit{dense image} of $(Y,\r S)$ if there is a dense $f$-morphism $g: (Y,\r S) \to (X,\r R)$. Moreover, $(X,\r R)$ is a \textit{strong subframe} (resp. a \textit{strongly dense image}) of $(Y,\r S)$ if $h$ is a strong embedding (resp. $g$ is strongly dense).
\end{definition}

\begin{lemma} \label{duallma}
    Let $(X,\r R)$ and $(Y,\r S)$ be two fundamental frames. Then:
    
    \begin{itemize}
\item  If $(X,\r R)$ is a subframe of $(Y,\r S)$, then $\pchi R(X)$ is a homomorphic image of $\pchi S(Y)$;
\item If $(X,\r R)$ is a dense image of $(Y, \r S)$, then $\pchi R(X)$ is a subalgebra of $\pchi S(Y)$. 
\end{itemize}

Moreover, for any two fundamental lattices $(L, \neg)$ and $(M, \resim)$:

\begin{itemize}
    \item If $(L, \neg)$ is a subalgebra of $(M, \resim)$, then $\digamma(L)$ is a strongly dense image of $\digamma(M)$;
    \item If $(L, \neg)$ is a homomorphic image of $(M, \resim)$, then  $\digamma(L)$ is a strong subframe of $\digamma(M)$.
    \end{itemize}
\end{lemma}

\begin{proof}
    The first part of the lemma follows directly from Lemmas \ref{injlma} and \ref{surjlma}, and the second part is an immediate consequence of Lemma \ref{replma2}.
\end{proof}

We conclude this section by identifying the frame-theoretic notion that corresponds to products of fundamental lattices. Unsurprisingly, this is given by a disjoint union construction.

\begin{definition}
    Let $\{(X_i,\r R_i)\}_{i \in I}$ be a family of fundamental frames. The \textit{coproduct} of the family $\{(X_i,\r R_i)\}_{i \in I}$ is the relation frame $(X_I, \r R_I)$, where $X_I$ is the disjoint union of the sets $X_i$, and the relation $\r R_I$ is given by $(i,x)\r R_I(i',x')$ iff $i = i'$ and $x\r R_ix'$ for any $x \in X_i$, $x' \in X_{i'}$.
\end{definition}

It is straightforward to verify that the coproduct of any family of fundamental frames is also a fundamental frame. In fact, the dual fundamental lattice of the coproduct of a family of frames is easily seen to be isomorphic to the product of the corresponding family of fundamental lattices.

\begin{lemma} \label{prodlma}
    For any family $\{(X_i,\r R_i)\}_{i \in I}$ of fundamental frames with coproduct $(X_I, \r R_I)$, $\chi_{\r R_I}(X_I)$ is isomorphic to $\prod_{i \in I} \chi_{\r R_i}(X_i)$.
\end{lemma}

\begin{proof}
    Define the map $f: \prod_{i \in I} X_i  \to X_I$ by $f(\{A_i\}_{i \in I}) = \bigcup_{i \in I} (i, A_i)$, where $(i,A_i) = \{(i, x) \in X_I \mid x \in A_i\}$. Since all relations $\r R_i$ are disjoint from one another, it is easy to check that for any $A \sset X_I$, $A \in \chi_{\r R_I}(X_I)$ iff $\{x \in X_i \mid (i, x) \in A\} \in \chi_{\r R_i}(X_i)$ for every $i \in I$. This shows that the image of the restriction of $f$ to $\prod_{i \in I} \chi_{\r R_i}(X_i)$ is exactly $\chi_{\r R_I}(X_I)$. Moreover, $f$ clearly preserves and reflects the inclusion order, and hence its restriction to $\prod_{i \in I} \chi_{\r R_i}(X_i)$ is an isomorphism. This completes the proof.
\end{proof}

 \section{The Goldblatt-Thomason Theorem} \label{section4}

 In this section, we prove our main result, namely, a version of the Goldblatt-Thomason Theorem for Fundamental Logic. In the original setting of modal logic, ultrafilter extensions play a central role in bridging the gap between algebraic and relational structures. In our setting, a similar role is played by filter extensions, to which we now turn.

 \begin{definition}
     Let $(X,\r R)$ be a fundamental frame. The \textit{filter extension} of $(X,\r R)$ is the fundamental frame $\digamma(\chi_R(X))$. 
 \end{definition}

 The filter extension of a fundamental frame $(X,\r R)$ can be characterized more concretely as follows. Points are pairs $(F,I)$ such that $F$ is a proper filter on the positive algebra $\pchi R(X)$ and $I$ is a proper filter on the negative algebra $\nchi R(X)$ such that $A \in F$ implies $\nneg R \pneg R A \in I$. Moreover, for any two such points $(F,I), (G,J)$, we let $(F,I)\r S(G,J)$ iff there is no $A \in G$ such that $\nneg R A \in I$. Note that, since $\pneg R$ and $\nneg R$ are inverse anti-isomorphisms between $\nchi R(X)$ and $\pchi R(X)$, this is equivalent to taking exactly the points in the canonical frame of $\pchi R (X)$.\\

 It is worth commenting on the relationship between the positive algebra of a fundamental frame and the positive algebra of its filter extension. As it turns out, the latter is the $\pi$-canonical extension of the former (see \cite{Dunn} for the general definition of $\pi$-canonical extensions of lattice expansions). This follows from the following general fact about fundamental lattices.

 \begin{lemma} \label{canlma}
     Let $(L, \neg)$ be a fundamental lattice. Then $\chi(\digamma(L))$ is the $\pi$-canonical extension of $L$, as witnessed by the embedding $\widehat{\cdot} : L \to \chi(\digamma(L))$.
 \end{lemma}

  \begin{proof}
     Recall first that the canonical extension of a lattice $L$ is characterized up to isomorphism as a complete lattice $L$ and an embedding $\alpha: L \to C$ with the following properties:

     \begin{itemize}
         \item $L$ is doubly-dense in $C$: for any $A \in C$, there are families $\{I_k\}_{k \in K}$ and $\{J_h\}_{h \in H}$ of subsets of $L$ such that $\bigme_C \{\bigjo_C \{\alpha(a) \mid a \in I_k \mid k \in K\} = A = \bigjo_C \{\bigme_C \{\alpha(b) \mid b \in J_h \mid h \in H\}$;
         \item $L$ sits compactly inside of $C$: for any $A, B \sset L$ such that $\bigme_C \{\alpha(a) \mid a \in A\} \leq_C \bigjo_C\{\alpha(b) \mid b \in B\}$, there are finite sets $A' \sset A$ and $B'\sset B$ such that $\bigme_L A' \leq_L \bigjo_L B'$.
     \end{itemize}

     Let $\digamma(L) = (X,\r R)$. We prove that the embedding $\widehat{\cdot} : L \to \pchi R(L)$ satisfies these two properties. Let $A \in \pchi R (L)$. First, we claim that $A = \bigcup_{x \in A} \bigcap_{a \in x_F} \widehat{a}$. Since $\pneg R \nneg R A = A$, this will establish that $A = \bigjo_{x \in A} \bigme_{a \in x_F} \widehat{a}$. For the proof of the claim, note that the left-to-right inclusion is immediate. For the converse, let $x \in X$ and suppose that there is $x' \in A$ such that $x' \in \bigcap_{a \in x_F} \widehat{a}$. Then $x' \ple R x$, so, since $A \in \pchi R$, it follows from Fact \ref{facts}.(vi) that $x' \in A$. This shows that any element in $\pchi R(X)$ is a join of meets of images of elements of $L$. To show that it is also a meet of joins of images of elements of $L$, note first that a completely similar argument shows that, for any $B \in \nchi R (X)$, $B = \bigjo_{\nchi R (X)} \{ \bigme_{\nchi R(X)} \{ \widecheck{b} \mid b \in x_I\} \mid x \in B\}$, where $\widecheck{b} = \{x \in X \mid b \in x_I\}$. Now for any $A \in \pchi R (X)$, there is $B \in \nchi R (X)$ such that $A = \pneg R B$. Fix $A \in \pchi R (X)$ and such a $B$. Since $\pneg R$ is an anti-isomorphism, it follows that $A = \bigme_{ \pchi R(X)} \{ \bigjo_{\pchi R (X)} \{ \pneg R \widecheck{b} \mid b \in x_I\} \mid x \in B\}$. Since $\pneg R \widecheck{b} = \widehat{b}$ for any $b \in L$, this completes the proof that $L$ is doubly-dense in $\pchi R (X)$.

     Let us now prove that $L$ sits compactly inside of $\pchi R (X)$. Let $A, B \sset L$ such that for any finite $A' \sset A$, $B' \sset B$, $\bigme_L A' \nleq_L \bigjo_L B'$. This means that there exist $F$ and $I$, respectively a filter and an ideal on $L$, such that $A \sset F$, $B \sset I$ and $F \cap I = \emptyset$. Let $x \in X$ be the point $(F,I)$. Note that $x \in \bigcap_{a \in A} \widehat{a}$. However, $x \r R x$, and for any $y \in X$ such that $x \r R y$, $y \notin \bigcup_{b \in B} \widehat{b}$, since otherwise $I \cap y_F \neq \emptyset$, contradicting $x \r R y$. Hence $x \notin \pneg R \nneg R (\bigcup_{b \in B} \widehat{b})$, which shows the compactness property.

     To complete the proof that $(\pchi R (\digamma(L), \pneg R)$ is the $\pi$-canonical extension of $(L,\neg)$, it remains to show that $\pneg R$ is the $\pi$-extension of $\neg$. Recall that, for $f: L \to L$ an antitone map, its $\pi$-extension $f^\pi:L^\sigma \to L^\sigma$ is given by:
     \[f^\pi(a) = \bigme \{\bigjo_{b \in F} f(b) \mid F \sset L:\text{ is a filter and }\bigme F \leq a\} \]
     for any $a \in L^\sigma$.

     In our setting this means that we must show that for any $A \in \pchi R (\digamma(L))$, $\pneg R A = \bigcap_{x \in A} \bigjo_{a \in x_F} \widehat{\neg a}$. For the left-to-right direction, fix some $y \in \pneg R A$ and some $x \in A$. It is enough to show that $\neg[x_F] \cap y_I \neq \emptyset$, where $\neg[x_F] = \{\neg a \mid a \in x_F\}$. Suppose towards a contradiction that $\neg[x_F] \cap y_I = \emptyset$. Then consider the pair $z=(x_F,\neg[x_F]) \in \digamma(L)$, and notice that $z\r Ry$. But $z \ple R x$, so $z \in A$. This contradicts $y \in \pneg R A$. For the converse direction, suppose that $y \notin \pneg R A$. Then there is $x \in A$ such that $x\r R y$. We claim that for all $z \reef{R} x$ and for all $a \in x_F$, $z \notin \widehat{\neg a}$. Indeed, if there is $a \in x_F$ such that $z \in \widehat{\neg a}$, then $z_F \cap x_I \neq \emptyset$, so $\neg x \r R z$. But this means that $x \in \nneg R \bigcup_{a \in x_F} \widehat{\neg a}$, hence $y \notin \pneg R \nneg R \bigcup_{a \in x_F} \widehat{\neg a} = \bigjo_{a \in x_F} \widehat{\neg a}$. This completes the proof.
     
 \end{proof}

The following is a standard definition in the literature on Goldblatt-Thomason theorems.

 \begin{definition}
     A class $\mathfrak{K}$ of fundamental frames is \textit{axiomatic} if there is a set $\Gamma$ of formulas of $FL$ such that $\mathfrak{K} = Mod(\Gamma)$.
 \end{definition}

 Note that a class $\mathfrak{K}$ is axiomatic iff $\mathfrak{K} \rset Mod(Log(\mathfrak{K}))$, as an easy argument shows. The following is a routine fact.

 \begin{lemma} \label{gtlma1}
     Let $\mathfrak{K}$ be an axiomatic class of fundamental frames. Then $\mathfrak{K}$ is closed under subframes, dense images and coproducts, and it reflects filter extensions.
 \end{lemma}

 \begin{proof}
     This is a routine argument. Suppose that $\mathfrak{K} = Mod(\Gamma)$ for some set $\Gamma$ of formulas of $\mathcal{L}$, and let $(X,\r R)$ and $(Y, \r S)$ be fundamental frames. If $(X,\r R)$ is a subframe (resp. dense image) of $(Y,\r S)$, then $\pchi R(X)$ is a homomorphic image of (resp. embeds into) $\pchi S(Y)$ by Lemma \ref{duallma}. Now if $(Y,\r S)$ is in $\mathfrak{K}$, then $\Gamma$ is valid on $\pchi S(Y)$. But then $\Gamma$ is also valid on any subalgebra or homomorphic image of $\pchi S(Y)$. This shows that $\mathfrak{K}$ is closed under subframes and dense images.
     
     Moreover, if $\{(X_i,\r R_i)\}_{i \in I}$ is a family of fundamental frames in $\mathfrak{K}$, then $\Gamma$ is valid on $\chi_{\r R_i} (X_i)$ for any $i \in I$, hence also on $\prod_{i \in I} \chi_{\r R_i} (X_i)$. By Lemma \ref{prodlma}, the latter is isomorphic to $\chi_{\r R_I}(X_I)$, the positive algebra of the coproduct $(X_I, \r R_I)$ of the family $\{(X_i,\r R_i)\}_{i \in I}$. Hence $(X_I, \r R_I) \in \mathfrak{K}$, which shows that $\mathfrak{K}$ is closed under coproducts. 
     
     Finally, assume that $(X,\r R)$ is a fundamental frame such that its filter extension $\digamma(\pchi R(X))$ is in $\mathfrak{K}$. This means that $\Gamma$ is valid on the positive algebra of $\digamma(\pchi R(X))$. But by Lemma \ref{canlma}, the latter is the canonical extension of $\pchi R (X)$, hence $\pchi R (X)$ embeds into it. It follows that $\Gamma$ is valid on $\pchi R (X)$, and therefore $(X,\r R) \in \mathfrak{K}$. Hence $\mathfrak{K}$ reflects filter extensions.
 \end{proof}

Our key lemma towards a Goldblatt-Thomason theorem is the following.

 \begin{lemma} \label{gtlma2}
     Let $\mathfrak{K}$ be a class of fundamental frames closed under filter extensions. If $\mathfrak{K}$ is closed under strong subframes, strongly dense images and coproducts, and it reflects filter extensions, then $\mathfrak{K}$ is axiomatic.
 \end{lemma}

 \begin{proof}
     Suppose $\mathfrak{K}$ satisfies the conditions of the lemma and let $\Gamma = Log(\mathfrak{K})$. Clearly $\mathfrak{K} \sset Mod(\Gamma)$, so we only need to show the converse. So assume that $(X,\r R)$ is a fundamental frame such that $\Gamma$ is valid on $(X,\r R)$. By Birkhoff's HSP theorem, it follows that $\pchi R (X) \in \mathbb{HSP}(\{\pchi S(Y) \mid (Y,\r S) \in \mathfrak{K}\})$. This means that there is a fundamental lattice $(L,\neg)$ and a family $\{(X_i,\r R_i)\}_{i \in I}$ of frames in $\mathfrak{K}$ such that $\pchi R (X)$ is a homomorphic image of $(L, \neg)$, and $(L,\neg)$ is a subalgebra of $\prod_{i \in I} \pchi {\r R_i} (X_i)$. By Lemma \ref{prodlma}, $(L, \neg)$ embeds into $\pchi {\r R_I} (X_I)$, where $(X_I, \r R_I)$ is the coproduct of the family $\{(X_i,\r R_i)\}_{i \in I}$. By Lemma \ref{duallma}, it follows that $\digamma(L)$ is a strongly dense image of $\digamma (\pchi {R_I} (X_I))$ and that $\digamma(\pchi R(X))$ is a strong subframe of $\digamma(L)$, as shown in the diagram below. 
     \vspace{0.5em}
     
      \adjustbox{width=0.8\textwidth,center}{%
 \begin{tikzcd}
	L && {\chi_{\r R_I}(X_I)} && {\digamma(L)} && {\digamma(\chi_{\r R_I}(X_I))} \\
	\\
	{\chi_{\r R}(X)} &&&& {\digamma(\chi_{\r R}(X))}
	\arrow[hook, from=1-1, to=1-3]
	\arrow[two heads, from=1-1, to=3-1]
	\arrow[two heads, from=1-7, to=1-5]
	\arrow[hook, from=3-5, to=1-5]
\end{tikzcd}}
\vspace{0.5em}

Now since $\mathfrak{K}$ is closed under coproducts and filter extensions, it follows that $\digamma (\pchi {R_I} (X_I)) \in \mathfrak{K}$. Since $\mathfrak{K}$ is closed under strongly dense images, $\digamma(L) \in \mathfrak{K}$, and since $\mathfrak{K}$ is closed under strong subframes, $\digamma(\pchi R(X)) \in \mathfrak{K}$. Finally, since $\mathfrak{K}$ reflects filter extensions, we can conclude that $(X,\r R) \in \mathfrak{K}$. Hence $\mathfrak{K}$ is axiomatic.
 \end{proof}

Recall that we identified in the previous section two distinct notions of subframes and two distinct notions of dense images. In the case of axiomatic classes, we can see that the two notions coincide in a sense spelled out by our main result, which we are now in a position to prove.

 \begin{theorem} \label{mainthm}
     Let $\mathfrak{K}$ be a class of fundamental frames closed under filter extensions. Then the following are equivalent:

     \begin{enumerate}
         \item $\mathfrak{K}$ is axiomatic;
         \item $\mathfrak{K}$ is closed under subframes, dense images and coproducts, and it reflects filter extensions;
         \item $\mathfrak{K}$ is closed under strong subframes, strongly dense images and coproducts, and it reflects filter extensions. 
     \end{enumerate}
\end{theorem}

     \begin{proof}
         We have the following chain of implications: \[(i)\Rightarrow (ii) \Rightarrow (iii) \Rightarrow (i),\] where the first implication follows from Lemma \ref{gtlma1}, the second implication is immediate, and the third one follows from Lemma \ref{gtlma2}.
     \end{proof}

Finally, we can also characterize classes of fundamental frames of the form $Mod(L)$ for $L$ a canonical superfundamental logic, i.e., a logic  extending Fundamental Logic whose corresponding variety of fundamental lattices is closed under $\pi$-canonical extensions.

\begin{corollary} \label{cancor}
    Let $\mathfrak{K}$ be a class of fundamental frames. The following are equivalent:

    \begin{enumerate}
        \item $\mathfrak{K} = Mod(L)$, for $L$ a canonical superfundamental logic;
        \item $\mathfrak{K}$ is closed under filter extensions, subframes, dense images and coproducts, and it reflects filter extensions;
        \item $\mathfrak{K}$ is closed under filter extensions, strong subframes, strongly dense images and coproducts, and it reflects filter extensions.
    \end{enumerate}
\end{corollary}

\begin{proof}
    The equivalence between $(ii)$ and $(iii)$ clearly follows from Theorem \ref{mainthm}, so we only need to check that $(i) \Leftrightarrow (ii)$. Suppose first that $\mathfrak{K}$ is axiomatic and that $L = Log(\mathfrak{K})$ is canonical. By Theorem \ref{mainthm}, it is enough to show that $\mathfrak{K}$ is closed under filter extensions. So suppose $(X,\r R) \in \mathfrak{K}$. Then $\pchi R (X)$ is in the variety corresponding to $L$. Since $L$ is canonical, this implies that $(\pchi R (X)^\sigma,\pneg R^\pi)$, the $\pi$-canonical extension of $(\pchi R (X), \pneg R)$, is also in that variety. But since the positive algebra of the filter extension of $(X,\r R)$ is isomorphic to $(\pchi R(X)^\sigma,\pneg R^\pi)$ by Lemma \ref{canlma}, it follows that $L$ is valid on $\digamma(\pchi R (X))$, hence $\digamma(\pchi R(X)) \in Mod(L) = \mathfrak{K}$ since $\mathfrak{K}$ is axiomatic. Hence $\mathfrak{K}$ is closed under filter extensions.
    
    Conversely, let us now assume that $\mathfrak{K}$ is closed under filter extensions, subframes, dense images and coproducts, and that it reflects filter extensions. By Theorem \ref{mainthm}, it follows that $\mathfrak{K}$ is axiomatic, i.e., $\mathfrak{K} = Mod(Log(\mathfrak{K}))$. It remains to show that $L = Log(\mathfrak{K})$ is canonical. Let $A$ be a fundamental lattice in the variety corresponding to $L$. Then there is a family $\{(X_i,\r R_i)\}_{i \in I}$ of frames in $\mathfrak{K}$ and a fundamental lattice $B$ such that $A$ is a homomorphic image of $B$ and $B$ embeds into $\prod_{i \in I} \pchi {R_i} (X_i)$, which is isomorphic to $\pchi {R_I} (X_I)$, the positive algebra of the coproduct of the family $\{(X_i,\r R_i)\}_{i \in I}$. This yields the following two diagrams, where the right one is obtained from the left one by applying successively the functors $\digamma$ and $\chi$:

\vspace{0.5em}
 \adjustbox{width=0.8\textwidth,center}{%
\begin{tikzcd}
	B && {\chi_{\r R_I}(X_I)} && {\chi(\digamma(B))} && {\chi(\digamma(\chi_{\r R_I}(X_I)))} \\
	\\
	A &&&& {\chi(\digamma(A))}
	\arrow[hook, from=1-1, to=1-3]
	\arrow[two heads, from=1-1, to=3-1]
	\arrow[hook, from=1-5, to=1-7]
	\arrow[two heads, from=1-5, to=3-5]
\end{tikzcd}}
\vspace{0.5em}

Now $\digamma(\chi_{R_I}(X_I))$ is the filter extension of the coproduct of the family $\{(X_i,\r R_i)\}_{i \in I}$, so since $\mathfrak{K}$ is closed under coproducts and filter extensions, it follows that $\digamma(\chi_{\r R_I}(X_I)) \in \mathfrak{K}$, hence $L$ is valid on $\chi(\digamma(\chi_{\r R_I}(X_I)))$. But since $\chi(\digamma(A))$ is a homomorphic image of a subalgebra of $\chi(\digamma(\chi_{\r R_I}(X_I)))$, $L$ is also valid on $\chi(\digamma(A))$. Finally, since $\chi(\digamma(A))$ is isomorphic to the $\pi$-canonical extension of $A$ by Lemma \ref{canlma}, it follows that $L$ is a canonical. This completes the proof.
\end{proof}

 \section{Adding Modal Operators} \label{section5}

In this final section, we generalize the results obtained in the previous section to Fundamental Modal Logic, which was recently developed by Holliday in \cite{holliday2024modal}. We will focus on Holliday's \textit{additive unified} setting, in which the two modalities $\Box$ and $\Diamond$ are dual notions (even though, just like in modal intuitionistic logic, the two modalities may not be interdefinable), and $\Diamond$ is an additive operator. Moreover, we will only consider the case in which a single pair $(\Box, \Diamond)$ is added to the language of fundamental logic, but generalizations to a polymodal setting are obvious. We start with the following definitions, which generalize those of fundamental lattices and fundamental frames to the modal setting in a natural way.

\begin{definition}
    A fundamental modal lattice is a tuple $(L, \neg, \Box, \Diamond)$ such that $(L, \neg)$ is a fundamental lattice, and $\Box$ and $\Diamond$ are unary operations on $L$ such that for any $a,b \in L$:
    \begin{itemize}
        \item $\Box(a \me b) = \Box a \me \Box b$, $\Box 1 = 1$;
        \item $\Diamond(a \jo b) = \Diamond a \jo \Diamond b$, $\Diamond 0 = 0$;
        \item $\Diamond \neg a \leq \neg \Box a$.
    \end{itemize}
\end{definition}

\begin{definition}
    An \textit{additive unified fundamental modal frame} (\textit{$AUFM$ frame} for short) is a tuple $(X,\r R,M)$ such that $(X,\r R)$ is fundamental frame and $M$ is a relation on $X$ satisfying the following conditions for any $x, y, z \in X$:

    \begin{itemize}
        \item $xMy\reef R z \Rightarrow \exists x'\reef R x \forall y' \r R x' \exists z': x' M z' \reef R z$;
        \item $xMy \r R z \Rightarrow \exists x' \r R x \forall y' \reef R x' \exists z': x' M z' \r R z$.
    \end{itemize}
\end{definition}

As shown in \cite{holliday2024modal}, given an additive unified fundamental modal frame $(X,\r R,M)$, the unary operation $\Box_M: \Po(X) \to \Po(X)$ given by $\Box_M A = \{x \in X \mid \forall y \in X: x M y \Rightarrow y \in A\}$ restricts to a map from $\pchi R(X)$ to $\pchi R(X)$. Moreover, one can define a ``pseudo-dual'' $\Diamond_M$ to $\Box_M$, given by $\Diamond_M A = \pneg R \Box_M \nneg R A$, which has the property that $\Diamond_M \pneg R A \sset \pneg R \Box_M A$ for any $A \in \pchi R(X)$. Finally, any fundamental modal lattice $(L, \neg, \Box, \Diamond)$ modally embeds into the positive algebra of the expansion of its dual canonical frame $\digamma(L)$ with a relation $M$ given by $xMx'$ iff $\Box a \in x_F$ implies $a \in x'_F$ and $\Diamond b \in x_I$ implies $b \in x'_I$, and this defines an $AUFM$ frame. \\

Let us now define the relevant notion of morphism between $AUFM$ frames.

\begin{definition} \label{aufmdef}
    Let $(X,\r R,M)$ and $(Y,\r S,N)$ be two $AUFM$ frames. An \textit{$AUFM$-morphism} is a f-morphism $h: (X,\r R) \to (Y,\r S)$ satisfying the following additional constraints for any $x, x' \in X$ and any $y \in Y$:

\begin{enumerate}
    \item $x M x' \Rightarrow h(x) N h(x')$;
    \item $h(x) N y \Rightarrow \exists x' \in X: x M x'$ and $y \dle S h(x')$,
\end{enumerate}
where $\dle S = \ple S \cap \nle S$.
\end{definition}

\begin{lemma}
    Let $h: (X,\r R,M) \to (Y,\r S,N)$ be an $AUFM$-morphism between two $AUFM$ frames. Then for any $A \in \pchi S(Y)$, $h\inv[\Box_N A] = \Box_M h\inv[A]$ and $h\inv[\Diamond_N A] = \Diamond_M h\inv [A]$.
\end{lemma}

\begin{proof}
    Let us first show that $h\inv[\Box_N A] = \Box_M h\inv[A]$ for any $A \in \pchi S(Y) \cup \nchi S(Y)$. Suppose first that $h(x) \in \Box_N A$, and let $xMx'$. Then by condition $(i)$ on $AUFM$-morphisms, $h(x)Nh(x')$, so $x' \in h\inv[A]$. This shows the left-to-right inclusion. For the converse, assume that $x M x'$ implies $h(x') \in A$, and let $y \in Y$ be such that $h(x)Ny$. By condition $(ii)$ on $AUFM$-morphisms, there is $x' \in X$ such that $xMx'$ and $y \dle S h(x')$. But this means that $h(x') \in A$ and therefore also $y \in A$, since $A \in \pchi S (Y) \cup \nchi S (Y)$ and we have that $y \ple S h(x')$ and $y \nle S h(x')$.

    Now, to show that $h\inv$ also preserves the $\Diamond$ operation, recall that, for any $A \in \pchi S(Y)$, $\Diamond_N A = \pneg S \Box_N \nneg S A$. Since $(Y,\r S,N)$ is additive we have that $\Box_N B \in \nchi S(Y)$ whenever $B \in \nchi S(Y)$. Hence for any $A \in \pchi S(Y)$, we have the following chain of identities:
    \[h\inv[\pneg S \Box_N \nneg S A] = \pneg R h\inv[\Box_N \nneg S A] = \pneg R \Box_M h\inv[\nneg S A] = \pneg R \Box_M \nneg R h\inv[A],\]
    where the first and third identities follow from the fact that $h$ is an $f$-morphism. This completes the proof.
    \end{proof}

Here again, our definition of an $AUFM$-morphism is general enough to capture all modal fundamental lattice homomorphisms.

    \begin{lemma} \label{modmorlma}
        Let $(L, \neg, \Box_L,\Diamond_L)$ and $(N,\resim,\Box_N,
        \Diamond_N)$ be fundamental modal lattices, and $f:L \to N$ a fundamental modal homomorphism. Then $\digamma(f): \digamma(N) \to \digamma(L)$ is an $AUFM$-morphism.
    \end{lemma}

        \begin{proof}
        By Lemma \ref{canmorlma}, we know already that $\digamma(f)$, given by $\digamma(f)(x) = (f\inv[x_F],f\inv[x_I])$ for any $x \in \digamma(M)$, is a f-morphism. Hence we only need to check the extra two conditions in Definition \ref{aufmdef}. Let $R_L$ and $M_L$ be the following relations on $\digamma(L)$:
        \begin{itemize}
            \item $x \r R_L y$ iff $x_I \cap y_F = \emptyset$;
            \item $x M_L y$ iff $\{\Box_L a \mid a \in x_F\} \sset y_F$ and $\{\Diamond_L b \mid b \in x_I\} \sset y_I$,
        \end{itemize}
        and let $\r R_N$ and $M_N$ be defined similarly on $\digamma(N)$. Suppose first that $xM_Ny$, and let $\Box_L a \in (\digamma(f)(x))_F$. Then $h(\Box_L a) = \Box_Mh(a) \in x_F$, which means that $h(a) \in y_F$. Similarly, $\Diamond_L b \in (\digamma(f)(x))_I$ implies that $h(b) \in y_I$. But this means that $\digamma(f)(x)M_L\digamma(f)(y)$.

        To check that the second condition also holds, suppose now that $\digamma(f)(x)Ny$ for some $x \in \digamma(N), y \in \digamma(L)$. Let $x'=(F,I)$, where $F=\{a \in N \mid \Box_N a \in x_F\}$ and $I=\{b \in N \mid \Diamond_N b \in x_I\}$. To verify that $x' \in \digamma(N)$, we only need to check that $a \in F$ implies $\resim a \in I$. But if $a \in F$, we have that $\Box_N a \in x_F$, hence $\resim\Box_N a \in I$. Since $\Diamond_N \resim a \leq_N \resim \Box_N a$, it follows that $\Diamond_N \resim a \in x_I$, hence $\resim a \in I$. Moreover, we clearly have $xM_N x'$. Finally, for any $a,b \in L$, we have that: \begin{align*}
            a \in (\digamma(f)(x))_F &\Leftrightarrow f(a) \in F \Leftrightarrow \Box_N f(a) \in x_F \Leftrightarrow \Box_L a \in h(x')_F \Rightarrow a \in y_F\\
            b \in (\digamma(f)(x))_I &\Leftrightarrow f(b) \in I \Leftrightarrow \Diamond_N f(a) \in x_I \Leftrightarrow \Diamond_L a \in h(x')_I \Rightarrow b \in y_I,
        \end{align*}
        from which we conclude that $y \dle {R_L} \digamma(f)(x')$.
    \end{proof}

    As a consequence, we may straightforwardly adapt our analysis of the duals of homomorphic images, subalgebras and products to the modal case.

    \begin{definition}
            Let $(X,\r R,M)$ and $(Y,\r S,N)$ be $AUFM$ frames. Then $(X,\r R,M)$ is a \textit{modal subframe} of $(Y,\r S,N)$ if there is an embedding $h: (X,\r R) \to (Y,\r S)$ which is also an $AUFM$-morphism, and it is a \textit{dense modal image} of $(Y,\r S,N)$ if there is a dense $f$-morphism $g: (Y,\r S) \to (X,\r R)$ which is also an $AUFM$-morphism. Moreover, $(X,\r R,M)$ is a \textit{strong modal subframe} (resp. a \textit{strongly dense modal image}) of $(Y,\r S,N)$ if $h$ is a strong embedding (resp. $g$ is strongly dense).
    \end{definition}

    \begin{definition}
        The coproduct of a family $\{(X_i,\r R_i,M_i)\}_{i \in I}$ of $AUFM$ frames is the frame $(X_I,\r R_I,M_I)$, where $(X_I,\r R_I)$ is the coproduct of the family of fundamental frames $\{(X_i,\r R_i)\}_{i \in I}$, and $M_I$ is the disjoint union of the relations $\{M_i\}_{i \in I}$.
    \end{definition}

    It is routine to verify that the coproduct of a family of $AUFM$ frames is itself an $AUFM$ frame, and that its positive algebra is isomorphic to the product of the positive algebras of each frame in the family, so we will omit the proof. Similarly, we define the filter extension of an $AUFM$ frame in a standard way.

    \begin{definition}
        Let $(X,\r R,M)$ be a $AUFM$ frame. Then the filter extension of $(X,\r R,M)$ is the $AUFM$ frame $(Y, \r S, N) = \digamma(\pchi R (X),\pneg R, \Box_M, \Diamond_M)$, where for any $x, y \in Y$, $xNy$ iff ($\Box_M A \in x_F$ implies $A \in y_F$, and $\nneg R \Diamond_M A \in x_I$ implies $\nneg R A \in y_I$ for any $A \in \pchi R(X)$).
    \end{definition}
    
Filter extensions correspond to canonical extensions in the modal case as well.

    \begin{lemma} \label{modcanlma}
        For any fundamental modal lattice $(L,\neg,\Box, \Diamond)$, the positive algebra of $\digamma(L, \neg, \Box, \Diamond)$ is the $\pi$-canonical extension of $(L, \neg, \Box, \Diamond)$.
    \end{lemma}

        \begin{proof}
        From Lemma \ref{canlma}, we already know that $(\pchi R(\digamma(L, \neg)), \nneg R)$ is the $\pi$-canonical extension of $(L, \neg)$. Hence we only need to verify that $\Box_M$ is the $\pi$-extension of $\Box$ and that $\Diamond_M$ is the $\pi$-extension of $\Diamond$. For the first one, we must show that for any $A \in \chi(\digamma(L))$, \[\Box_M A = \bigcap\{ \pneg R \nneg R \bigcup_{a \in I} \widehat{\Box a} \mid I \sset L \text{ is an ideal and } A \sset \pneg R \nneg R \bigcup_{a \in I} \widehat{a}\}.\]

        First, we claim that for any ideal $I$ on $L$, $A \sset \bigjo_{a \in I} \widehat{a}$ iff $x_F \cap I \neq \emptyset$ for all $x \in A$. For the left-to-right direction, if there is $x \in A$ such that $x_F \cap I = \emptyset$, then let $y= (\{1\}, I)$. We have that $y \r R x$, but clearly, if $y \r R z$, then $z \notin \bigcup_{a \in A} \widehat{a}$, hence $x \notin \bigjo_{a \in I} \widehat{a}$. For the right-to-left direction, if $x_F \cap I \neq \emptyset$ for all $x \in A$, then $A \sset \bigcup_{a \in I} \widehat{a} \sset \bigjo_{a \in I} \widehat{a}$. This completes the proof of the claim.
        
        Now for the left-to-right direction of the equality, suppose that $x \in \Box_M A$, and let $y = (\{a \mid \Box a \in x_F\}, \{b \mid \Diamond b \in x_I\})$. Clearly, $xMy$, hence $y \in A$. But by the claim, this means that $y_F \cap I \neq \emptyset$. Hence $x \in \widehat{\Box a}$ for some $a \in I$. For the right-to-left direction, now suppose that $x \notin \Box_M A$. Then we have $y, z$ such that $x M y \r R z$ and $z \in \nneg R A$. Since $z \in \nneg R A$, this means that for any $w \in A$, $z_I \cap w_F \neq \emptyset$, and therefore, by the claim, $A \sset \bigjo_{a \in z_I} \widehat{a}$. Now it only remains to show that $x \notin \bigjo_{a \in z_I} \widehat{\Box a}$. Let $w = (\{1\}, I)$, where $I = \dnset\{\Box b \mid b \in z_I\}$ (note that $I$ is an ideal because $\Box a \jo \Box b \leq \Box (a \jo b)$ for any $a, b \in z_I$). Clearly, $w \in \nneg \bigcup_{a \in z_I} \widehat{\Box a}$. But we also have $x\r R w$: indeed, if $x_F \cap w_I \neq \emptyset$, then there is $b \in z_I$ such that $\Box b \in x_F$, and hence $b \in y_F$, contradicting our assumption. This shows that $x \notin \bigjo_{a \in z_I} \widehat{\Box a}$.

        Finally, let us show that $\Diamond_M$ is the $\pi$-extension of $\Diamond$. For any $A \in \pchi R (\digamma(L))$ and any $x \in \digamma(L)$, we need to show that $x \in \Diamond_M A$ iff $x \in \bigjo_{a \in I} \widehat{\Diamond a}$ for every ideal $I$ such that $A \sset \bigjo_{a \in I} \widehat{a}$. For the left-to-right direction, suppose that $x \in \Diamond_M A$ and that $I$ is an ideal such that $A \sset \bigjo_{a \in I} \widehat{a}$. Using the fact that $\Diamond_M = \pneg R \Box_M \nneg R$, a simple computation shows that we need to show that $x \in \pneg R \bigcap_{a \in I} \Box_M \widecheck{a}$. For this, it is enough to show that $\dnset \{\Diamond a \mid a \in I\} \cap x_F \neq \emptyset$, since this will show that for any $y \r R x$, $y \notin \Box_M \widecheck{a}$ for some $a \in I$. Suppose, towards a contradiction, that $\dnset \{\Diamond a \mid a \in I\} \cap x_F = \emptyset$. Then there is $y\r Rx$ such that $y_I = \dnset \{\Diamond a \mid a \in I\}$. But now we claim that $y \in \Box \nneg R A$. Indeed, if $y M z \r R w$, then $w_F \cap I = \emptyset$ which, by the first claim above, means that $w \notin A$. Hence $x \notin \pneg R \Box_M \nneg R A$, contradicting the assumption that $x \in \Diamond_M A$.

        For the converse direction, suppose now that $x \notin \Diamond_M A = \pneg R \Box_M \nneg R A$. Then we have $y \in \Box_M \nneg R A$ such that $y\r Rx$. Let $w = (\{a \mid \Box a \in y_F\}, \{b \mid \Diamond b \in y_I\})$. Then $yMw$, hence $w \in \nneg R A$. But this implies that $w_I \cap z_F = \emptyset$ for any $z \in A$, hence, by the first claim above, $A \sset \bigjo_{a \in w_I} \widehat{a}$. Now we claim that $x \notin \bigjo_{a \in w_I} \widehat{\Diamond a}$. To show this, it is enough to have that $x_F \cap \{\Diamond a \mid a \in w_I\} = \emptyset$. But the latter is clear, since $a \in w_I$ implies $\Diamond a \in y_I$, which in turns implies $\Diamond a \notin x_F$ by assumption on $y$. This completes the proof.
    \end{proof}

As a direct consequence of Theorem \ref{mainthm}, we obtain a Goldblatt-Thomason theorem for fundamental modal logic, as well as a characterization of classes of $AUFM$ frames of the form $Mod(L)$ for $L$ a canonical superfundamental modal logic.

 \begin{theorem} \label{modalthm}
     Let $\mathfrak{K}$ be a class of $AUFM$ frames closed under filter extensions. Then the following are equivalent:

     \begin{enumerate}
         \item $\mathfrak{K}$ is axiomatic;
         \item $\mathfrak{K}$ is closed under modal subframes, dense modals images and coproducts, and it reflects filter extensions;
         \item $\mathfrak{K}$ is closed under strong modal subframes, strongly dense modal images and coproducts, and it reflects filter extensions. 
     \end{enumerate}
\end{theorem}

\begin{corollary}\label{modcancor}
    Let $\mathfrak{K}$ be a class of $AUFM$ frames. The following are equivalent:

    \begin{enumerate}
        \item $\mathfrak{K} = Mod(L)$, for $L$ a canonical superfundamental modal logic;
        \item $\mathfrak{K}$ is closed under filter extensions, modal subframes, dense modal images and coproducts, and it reflects filter extensions;
        \item $\mathfrak{K}$ is closed under filter extensions, strong modal subframes, strongly dense modal images and coproducts, and it reflects filter extensions.
    \end{enumerate}
\end{corollary}

The proofs of these two results completely mirror the non-modal case, and are therefore omitted.

\section*{Acknowledgements}

I thank Wes Holliday for helpful discussions on the topic, as well as three anonymous referees for their comments which helped improve the clarity of the paper.

\bibliographystyle{aiml}
\bibliography{aiml}

\end{document}